\documentclass{amsart}
\usepackage{amsfonts}

\newtheorem{theorem}{Theorem}
\theoremstyle{plain}

\newtheorem{corollary}{Corollary}

\newtheorem{lemma}{Lemma}

\newtheorem{remark}{Remark}

\numberwithin{equation}{section}

\input{tcilatex}

\begin{document}
\title[KC spaces]{On low dimensional KC-spaces}
\author{Paul Fabel}
\address{Drawer MA, Department of Mathematics and Statistics, Mississippi
State University, MS 39762}
\email{fabelATra.msstate.edu}
\date{February 19 2011}

\begin{abstract}
The KC property, a separation axiom between weakly Hausdorff and Hausdorff,
requires compact subsets to be closed. Various assumptions involving local
conditions, dimension, connectivity, and homotopy show certain KC-spaces are
in fact Hausdorff. Several low dimensional examples of compact, connected,
non-Hausdorff KC-spaces are exhibited in which the nested intersection of
compact connected subsets fails to be connected.
\end{abstract}

\maketitle

\section{\protect\bigskip Introduction}

What are the strongest properties permitted of a space $X,$ if the nested
intersection of compact connected subsets of $X$ can fail to be connected? $%
X $ must be non-Hausdorff, but examples otherwise rich in structure are not
so obvious since for example, standard texts such as \cite{Munkres} devote
little attention to such spaces.

To weaken slightly the Hausdorff condition, we turn to \textbf{KC-spaces},
spaces in which compact subspaces are closed. We construct a variety of
non-Hausdorff KC-spaces in which that nested intersection of closed compact
connected subspaces fails to be connected. We establish two theorems which
eliminate certain classes of non-Hausdorff $KC$ examples, and finally we
observe that the paper's content applies to \textbf{WH, }the category of
weakly Hausdorff spaces.

For historical background, Example 99 \cite{Steen} shows why KC-spaces can
be non-Hausdorff. \ More generally a technique for constructing KC-spaces is
implicit in the 1967 paper by Wilansky \cite{Wil} which shows the
Alexandroff compactification of a $k+$KC-space is again a KC-space.
KC-spaces are also called \textbf{maximal compact} spaces, and in certain
contexts such spaces are guaranteed to be Hausdorff. For example the 1985
the paper of A.H. Stone \cite{Stone} shows that 1st countable maximal
compact spaces are in fact $T_{2}.$ The general theory of KC-spaces has
continued to develop over the last decade \cite{Nyikos} \cite{Sun} and
remains an active area of research. Recent advances include a proof of a
long standing conjecture that minimal $KC$-spaces are compact \cite{Bella},
and several questions posed in \cite{myst} are settled in \cite{Bald}. 

In the paper at hand Theorem \ref{t2rite} shows every simply connected,
locally path connected, 1-dimensional KC-space is $T_{2}$. Theorem \ref%
{t2roid} shows the Hausdorff property is also guaranteed in KC-spaces which
are locally connected by continua and in which compact connected subsets
always have connected intersection. Corollaries \ref{ritecor} and \ref%
{dendroid} yield new criteria by which dendrites and certain dendroids can
be recognized.

The remainder of the paper demonstrates, via example, the futility of
weakening the hypotheses of Theorems \ref{t2rite} and \ref{t2roid}. We
exhibit a series of non-Hausdorff examples typically constructed as the
Alexandroff compactification $W\cup \{\infty \}$ of a $T_{2}$ space $W$ such
that $W$ fails to be locally compact at precisely one point. Subsections \ref%
{1n}, \ref{1uni}, \ref{1cont}, and \ref{1lcc} exhibit 1-dimensional
counterexamples. Subsections \ref{2d} and \ref{2dnice} exhibit non-Hausdorff
2-dimensional KC-spaces, the latter of which is both contractible and
locally contractible.

The general relevance of $KC$-spaces is bolstered by the following
observations. Every $KC-$space $X$ is \textbf{weakly Hausdorff  }, (i.e.
maps from Hausdorff compacta into $X$ have closed image in $X$\textbf{). }In
various contexts, weakly Hausdorff spaces are better behaved than Hausdorff
spaces \cite{St}. For example, as noted (\cite{May2} (p. 485)) in reference
to Peter May's book \textquotedblleft The geometry of iterated loops
spaces\textquotedblright \cite{May2} (p. 485) \textquotedblleft The weak
Hausdorff rather than the Hausdoff property should be required....in order
to validate some of the limit arguments used in \cite{May1}%
.\textquotedblright\ In particular the Hausdorff property fails in general
to be preserved (with the direct limit topology) by the union of closed
Hausdorff subspaces \cite{Herrlich}. On the other hand, by definition, the $%
KC$ \ property is presevered under unions (with the direct limit topology)
of closed $KC$-subspaces.

This paper suggests two open questions. Is the $n-$connected example from
subsection \ref{2d} contractible? Does there exist a 1-dimensional
contractible non-Hausdorff KC-space?

\section{Definitions}

\bigskip A \textbf{continuum }is a compact connected Hausdorff (T$_{2}$)
space, and in particular we allow that a continuum is not metrizable.

A \textbf{Peano continuum} is a compact, connected, locally path connected
metrizable space.

The continuum $Y$ is \textbf{hereditarily unicoherent }if $C\cap D$ is
connected whenever $C$ and $D$ are subcontinua of $Y.$

The space $Y$ is \textbf{generalized hereditarily unicoherent }if $C\cap D$
is connected whenever $C$ and $D$ are closed compact connected subspaces of $%
Y.$

A space $X$ is a \textbf{KC-space} if each compact subspace of $X$ is closed
in $X$.

A space $X$ is a \textbf{k-space} if $A$ is closed in $X$ whenever $A\cap K$
is closed for all compact closed sets $K\subset X.$

An \textbf{arc }is a space homeomorphic to $[0,1].$

A \textbf{dendroid }is a hereditarily unicoherent continuum. (Note every
dendroid is $T_{2}$ but not necessarily metrizable).

A \textbf{dendrite} is a locally path connected continuum which contains no
simple closed curve. (Note every dendrite is $T_{2}$ but not necessarily
metrizable).

If $(X,\mathcal{T}_{X})$ and $\infty \notin X$ and if $Y=X\cup \{\infty \}$
then the \textbf{Alexandroff compactification of }$X$ is the space $(Y,%
\mathcal{T}_{Y})$ such $\mathcal{T}_{Y}$ is union of $\mathcal{T}_{X}$ and
sets $V$ such that $V\subset Y$ and $Y\backslash V$ is compact and closed in 
$X.$

The space $X$ is \textbf{locally compact} if for each $x\in X$ there exists
an open set $U$ such that $x\in U$ and $\overline{U}$ is compact.

The space $X$ has (covering) \textbf{dimension n,} if $n$ is minimal such
that for each open covering $\mathcal{G}$ of $X$ there exists an open
covering $\mathcal{G}_{1}$ of $X$ such that if $U\in \mathcal{G}_{1}$ there
exists $V\in \mathcal{G}$ such that $U\subset V,$ and such that each element 
$x\in X$ belongs to at most $n+1$ distinct sets of the collection $\mathcal{G%
}_{1}.$

The space $X$ is \textbf{connected by continua }if for each pair $%
\{x,y\}\subset X$ there exists a continuum $Z$ and a map $f:Z\rightarrow X$
such that $\{x,y\}\subset im(f).$ The space $X$ is \textbf{locally connected
by continua (}denoted \textbf{lcc) }if\textbf{\ }for each $x\in X$ there
exists an open set $U$ such that $x\in U$ and $U$ is connected by continua.

\begin{remark}
Familiar examples (such as the Alexandroff compactification of the rationals
discussed in subsection \ref{alexq}) show that `connected' is strictly
weaker than `connected by continua'. The property `locally connected by
continua' is at least as strong as `locally connected'. However if $X$ is
locally connected, it need not be the case that maps of ordered continua are
adequate to connect distinct points of $X$ \cite{Cornette} \cite{Mardesic}.
\end{remark}

\section{Promoting KC-spaces to Hausdorff}

The main results of this section are Theorems \ref{t2rite} and \ref{t2roid},
Corollaries \ref{ritecor} and \ref{dendroid}, which establish conditions
under which certain KC-spaces are necessarily Hausdorff.

\begin{lemma}
\label{kcez}Suppose $Y$ is a KC-space. If $Z$ is a compact $T_{2}$ space and 
$f:Z\rightarrow Y$ is continuous then $im(f)$ is a $T_{2}$ subspace of $Y.$
Each path connected subspace of $Y$ is arcwise connected.
\end{lemma}

\begin{proof}
Suppose $f:Z\rightarrow Y$ is a map and $Z$ is a compact $T_{2}$ space.
Suppose $\{a,b\}\subset im(f)$ and $a\neq b.$ Let $A=f^{-1}(\{a\})$ and $%
B=f^{-1}(\{b\}).$ Since $Z$ is compact and $T_{2},$ $Z$ is normal. Apply
normality of $Z$ to obtain disjoint open sets $U$ and $V$ in $Z$ such that $%
A\subset U,$ $B\subset V$ and $U\cap V=\emptyset .$ Let $K=Z\backslash U$
and $C=Z\backslash V.$ Then $Z=K\cup C$ and thus $im(f)=f(K)\cup f(C).$
Moreover, since each of $f(K)$ and $f(C)$ is compact, and since $Y$ is a
KC-space, each of $f(K)$ and $f(C)$ is closed in $Y.$ Thus $im(f)\backslash
f(K)$ and $im(f)\backslash f(C)$ establish the $T_{2}$ property of $im(f).$

Suppose $A\subset Y$ and $A$ is path connected and $\{a,b\}\subset A$ \ and $%
a\neq b.$ Obtain a path $\alpha :[0,1]\rightarrow A$ such that $\alpha (0)=a$
and $\alpha (1)=b.$ Since $[0,1]$ is a compact $T_{2}$ space, $im(\alpha )$
is a path connected $T_{2}$ subspace of $A$ and hence $im(\alpha )$ is
arcwise connected. Thus $A$ is arcwise connected.
\end{proof}

\begin{lemma}
\label{scc}Suppose $X$ is a 1-dimensional KC-space. Then $X$ is aspherical,
and moreover $X$ is simply connected if and only if $X$ contains no simple
closed curve.
\end{lemma}

\begin{proof}
To see that $X$ is aspherical suppose $f:S^{n}\rightarrow X$ is a map and $%
n\geq 2.$ Then $im(f)$ is $T_{2}$ by Lemma \ref{kcez}. By the Hahn
Mazurkiewicz theorem \cite{Maz} $im(f)$ is a 1 dimensional Peano continuum,
and hence $im(f)$ is aspherical (Cor. p578 \cite{Fort}). Thus $X$ is
aspherical.

If $X$ is simply connected then, to obtain a contradiction, suppose $X$
contains a simple closed curve $S\subset X.$ Since $X$ is simply connected
there exists a map $f:D^{2}\rightarrow X$ such that $f_{\partial D^{2}}$ is
an embedding onto $S.$ By the Hahn Mazurkiewicz theorem and lemma \ref{kcez} 
$im(f),$ is a one dimensional Peano continuum. Hence $S$ is a retract of $%
im(f)$ (Thm 3.1 \cite{Conner}), and hence the loop $S$ is both essential and
inessential in $im(f)$ and we have a contradiction.

Conversely suppose $X$ contains no simple closed curve and $f:\partial
D^{2}\rightarrow X$ is any map. By the Hahn Mazurkiewicz theorem and lemma %
\ref{kcez} $im(f)$ is a 1 dimensional Peano continuum which contains no
simple closed curve and thus $im(f)$ is contractible (Thm. p. 578 \cite{Fort}%
) and in particular $f$ is inessential in $X$. Thus $X$ is simply connected.
\end{proof}

\begin{theorem}
\label{t2rite}Suppose $X$ is a locally path connected KC-space and suppose $X
$ contains no simple closed curve. Then $X$ is Hausdorff.
\end{theorem}

\begin{proof}
Suppose $\{a,b\}\subset X$ and $a\neq b.$ Since $X$ is locally path
connected the components of $X$ are open and thus if $a$ and $b$ belong to
distinct components $Y_{a}$ and $Y_{b},$ the open sets $Y_{a}$ and $Y_{b}$
separate $a$ and $b.$

For the remaining case suppose $Y$ is a component of $X$ and $\{a,b\}\subset
Y.$ Then $Y$ is path connected since $Y$ is connected and locally path
connected. Thus $Y$ is arcwise connected by Lemma \ref{kcez}, and let the
arc $\alpha \subset Y$ have endpoints $\{a,b\}.$ Let $k\in \alpha \backslash
\{a,b\}.$ We claim $a$ and $b$ belong to distinct components $U$ and $V$ of $%
Y\backslash \{k\}.$

To obtain a contradiction suppose $U$ is a component of $Y\backslash \{k\}$
and $\{a,b\}\subset U.$ Since $U$ is open, $U$ is locally path connected.
Thus $U$ is path connected since $U$ is connected. By Lemma \ref{kcez} there
exists an arc $\beta \subset U$ with endpoints $\{a,b\}.$ Note $k\notin
\beta .$ Let $J$ denote the component of $\alpha \backslash \beta $ such
that $k\in J.$ Let $\{x,y\}$ denote the endpoints of $J$ in $\alpha .$ Let $%
I $ denote the component of $\beta \backslash \{x,y\}$ with endpoints $%
\{x,y\}. $ Observe $J\cup I\cup \{x,y\}$ is a simple closed curve and we
have a contradiction.

Let $U$ and $V$ denote the components of $Y\backslash \{k\}$ such that $a\in
U$ and $b\in V.$ Recall $Y$ is open in $X$ and $\{k\}$ is closed in $X.$
Thus $U$ and $V$ are open in $X$ and this proves $X$ is $T_{2}.$
\end{proof}

Combining Theorem \ref{t2rite} and Lemma \ref{scc} we obtain the following.

\begin{corollary}
\label{ritecor}If $X$ is a locally path connected, simply connected,
1-dimensional KC-space then $X$ is $T_{2},$ hence if $X$ is also compact and
connected then $X$ is a dendrite.
\end{corollary}

By definition, every hereditarily unicoherent continuum contains no simple
closed curve, and every locally path connected space is $lcc.$ Pairwise
replacement of the corresponding notions in the hypothesis of Theorem \ref%
{t2rite} yields the following theorem.

\begin{theorem}
\label{t2roid} Suppose the KC-space $X$ is generalized hereditarily
unicoherent and suppose $X$ is locally connected by continua$.$ Then $X$ is
Hausdorff.
\end{theorem}

\begin{proof}
Suppose $\{a,b\}\subset X$ and $a\neq b.$ Since $X$ is lcc, Lemma \ref{open}
ensures the components of $Y$ are open. Thus if $A$ and $B$ belong to
distinct components $Y_{a}$ and $Y_{b}$ then the open sets $Y_{a}$ and $%
Y_{b} $ separate $a$ and $b.$

Suppose $\{a,b\}$ belong to some component $Y\subset X.$ Then since $Y$ is
connected and lcc, $Y$ is connected by continua by Lemma \ref{lcc}. Obtain a
continuum $Z$ and a map $f:Z\rightarrow Y$ such that $\{a,b\}\subset im(f).$
Since $Y$ is a KC-space and $Z$ is a continuum $im(f)$ is a continuum by
Lemma \ref{kcez}.

Let $A$ denote the collection of all compact connected sets in $Y$ which
contain $\{a,b\}$ and let $B$ denote the collection of all subcontinua of $%
im(f)$ which contain $\{a,b\}.$

Let $\alpha $ denote the intersection of all sets in $A$ and let $\beta $
denote the intersection of all sets in $B.$Then $\alpha \subset \beta $
since $B\subset A$. On the other hand if $\gamma \in A$ then $im(f)\cap
\gamma \in A\cap B$ and hence $\beta \subset \alpha .$ Thus $\alpha =\beta .$
Note $\beta $ is a continuum by Lemma \ref{uni}.

Since $Y$ is $T_{1},$ $\{a,b\}$ is not connected and there exists $k\in
\beta \backslash \{a,b\}.$ Now we claim $a$ and $b$ belong to distinct
components $U$ and $V$ of $Y\backslash \{k\}.$ To obtain a contradiction
suppose there exists a component $U\subset Y\backslash \{k\}$ such that $%
\{a,b\}\subset U.$ Then $U$ is open since $Y$ is $T_{1}.$

Since $U$ is connected and $lcc,$ $U$ is connected by continua (by Lemma \ref%
{lcc}.) In particular there exists a compact connected set $\gamma \subset U$
such that $\{a,b\}\subset \gamma .$ Note $\gamma \in A$ and hence $\alpha
\subset \gamma .$ On the other hand $k\in \alpha \backslash \gamma $ and we
have a contradiction.

Thus $a$ and $b$ belong to distinct components $U$ and $V$ of $Y\backslash
\{k\}.$ Since $Y$ is open in $X$ and since $\{k\}$ is closed in $X$ it
follows that $U$ and $V$ are open in $X$ and this proves $X$ is Hausdorff.
\end{proof}

This yields immediately alternate criteria for dendroid recognition.

\begin{corollary}
\label{dendroid}Suppose the compact KC-space $X$ is connected, generalized
hereditarily unicoherent and suppose $X$ is locally connected by continua$.$
Then $X$ is a dendroid.
\end{corollary}

\section{Examples of connected,compact, non-Hausdorff KC-spaces}

The examples in the following subsections are compact, connected,
non-Hausdorff KC-spaces, and each contains a nested sequence of compact
connected subspaces $...A_{3}\subset A_{2}\subset A_{1}$ such that $\cap
_{n=1}^{\infty }A_{n}$ is not connected$.$

To construct such examples we begin with a 1st countable, Hausdorff space $X,
$ such that $X$ fails to be locally compact, and then manufacture the
Alexandroff compactification $X\cup \{b\}.$ Since $X$ is both a KC-space and
a $k$ space, Theorem 5 of \cite{Wil} ensures that $Y$ is a connected,
non-Hausdorff, compact KC-space. See Lemma \ref{alex} for an alternate
argument.

The examples also serve to illustrate how slight weakening of the hypotheses
in Theorems \ref{t2rite} and \ref{t2roid} can destroy the guarantee of the $%
T_{2}$ condition.

\subsection{The Alexandroff compactification of the rationals\label{alexq}}

To reinforce the relevance of the remaining examples in this paper, we begin
with a discussion of a well studied space which fails to enjoy most of the
properties of interest elsewhere in the paper at hand. Let $Y=Q\cup \{\infty
\}$ denote the Alexandroff compactification of the rational numbers $Q$ (see
also \cite{Steen})$.$ Lemma \ref{alex} ensures $Y$ is a compact
non-Hausdorff KC-space. The space $Y$ is not locally connected, not
connected by continua, not generalized hereditarily unicoherent, and $\dim
(Y)=1,$ argued as follows.

Since $Q$ is not connected Lemma \ref{alex} does not guarantee that $Y$ is
connected. To see why $Y$ is connected, observe if $A\subset Q$ and $A$ is
compact, then $\overline{Q\backslash A}=Y$ and in particular $Y$ cannot be
the disjoint union of two nonempty compact subspaces.

For similar reasons, if $U\subset Q$ and $U$ is open, then $\overline{U}$ is
connected in $Y,$ since if $A\subset U$ and $A$ is compact, then $\overline{%
U\backslash A}=\overline{U}.$ Thus if $Q_{+}$ and $Q_{-}$ denote the
positive and negative rationals then $\overline{Q_{+}}\cap \overline{Q_{-}}%
=\{0,\infty \}$ and thus $Y$ is not generalized hereditarily unicoherent.

Every nontrivial continuum is uncountable and hence has constant image in $%
Y. $ Thus $Y$ is not connected by continua.

To see that $\dim (Y)=1,$ given an open covering $\mathcal{G}$ of $Y$ let $%
\infty \in V$ and note $C=Y\backslash V$ is a compact zero dimensional set
of real numbers. Thus there exists a covering of $C$ by pairwise disjoint
open sets $\{U_{1},.,U_{n}\}$ subordinate to $\mathcal{G}\backslash V.$ Each
point of $Y$ belongs to at most 2 sets in $\{V,U_{1},..,U_{n}\}$ and hence $%
\dim (Y)\leq 1$ (and $\dim (Y)>0$ since $Y$ is not $T_{2}$).

Notice if $A_{n}=\overline{(0,\frac{1}{n})\cap Q}$ then $A_{n}$ is compact
and connected but $\cap _{n=1}^{\infty }A_{n}=\{0,\infty \}$, and the latter
set is not connected.

\subsection{1-dimensional and n-connected for all n\label{1n}}

In the previous example $Q\cup \{\infty \}$\bigskip\ is not path connected.
For the current example $Z$ is 1-dimensional, compact, non-Hausdorff,
connected and $n$-connected for all $n$. Corollary \ref{ritecor} demands
that $Z$ is not locally path connected and in this example $Z$ is not
locally connected.

Consider the following planar set $T\subset R^{2}$ such that $T=((-\infty
,\infty )\times \{0\})\cup (\{0\}\times \lbrack 0,\infty )).$ Note $T$ is a
metrizable space and in particular $T$ is Hausdorff and 1st countable. Let $%
\mathcal{A}$ denote the following collection of open subsets of $T.$ For
each positive integer $n$ consider the open subspace $A_{n}=((n,\infty
)\times \{0\})\cup (\{0\}\times \cup _{k=n}^{\infty }(2k,2k+1)).$

Let $\mathcal{A=}\{A_{1},A_{2},...\}$ and note $\mathcal{A}$ is countable, $%
A_{n+m}\cap A_{n}=A_{n+m}$ and if $x\in T$ there exist open sets $U$ and $%
A_{n}$ such that $x\in U$ and $A_{n}\cap U=\emptyset .$ Thus if $\mathcal{T}%
_{T}$ denotes the open sets of $T$ and if $\mathcal{S}_{a}$ denotes the sets
of the form $\{a\}\cup A$ for $A\in \mathcal{A}$, and if $Y=X\cup \{a\}$ is
the space with topology generated by $\mathcal{T}_{T}\cup \mathcal{S}_{a},$
then Lemma \ref{w1} ensures that $Y$ is a connected 1st countable $T_{2}$
space.

To check that $Y$ is not locally compact at $a,$ consider the closure of a
basic open set $C=\overline{\{a\}\cup A_{n}}$ and note $\{2n,2n+2,...\}%
\subset C$ and this sequence has no subsequential limit in $Y.$

Thus if $Z=Y\cup \{b\}$ denotes the Alexandroff compactification of $Y,$
then Lemma \ref{alex} ensures $Z$ is a compact, connected, non-Hausdorff
KC-space.

Note $T$ is path connected. Let $[-\infty ,\infty ]$ denote the two point
compactification of $(-\infty ,\infty ).$ For $N\geq 0$ define $%
j_{N}:([-\infty ,\infty ]\times \{0\})\cup (\{0\}\times \lbrack
0,N])\rightarrow Z$ via $j(-\infty ,0)=b,$ $j(\infty ,0)=a$ and $j(x)=x$
otherwise. \ By construction $j_{N}$ is continuous and one to one, hence by
Lemma \ref{kcez} $j_{N}$ is an embedding. Hence $im(j_{N})$ is contractible.
Note $\{a,b\}\subset $ $im(j_{N})$ and $(0,0)\in T\cap im(j_{N}).$ Thus $Z$
is path connected.

Observe if $\alpha :[0,1]\rightarrow Z$ is a map such that $\alpha (0)\neq a$
and $\alpha (1)=a$ then there exists $N$ such that $im(\alpha )\subset
\{a,b\}\cup ((-\infty ,\infty )\times \{0\})\cup (\{0\}\times \lbrack 0,N]).$
To see that $Z$ is $n-$connected, suppose $f:S^{n}\rightarrow Z$ is a map.
Since $S^{n}$ is a Peano continuum obtain a surjective map $\beta
:[0,1]\rightarrow S^{n}.$ Let $\alpha =f\beta $ and note $im(f)=im(\alpha ).$
Thus $im(f)$ is contained in the contractible subspace $\{a,b\}\cup
((-\infty ,\infty )\times \{0\})\cup (\{0\}\times \lbrack 0,N]).$ Hence $f$
is inessential.

To check that $\dim (Y)=1$, take an open covering $\mathcal{G}$ of $Y$ with $%
a\in U$ and $b\in V.$ Replace $U$ and $V$ by basic open sets $U_{1}\subset U$
and $V_{1}\subset V$ such that $Y\backslash (U_{1}\cup V_{1})=([-N,N]\times
\{0\})\cup (\{0\}\times \lbrack 0,N])$ and such that $\{(N,0)\cup
(0,N)\}\subset U_{1}\backslash V_{1}$ and $(-N,0)\in V_{1}\backslash U_{1}$
and manufacture a covering $\mathcal{G}_{1}$ of $Y$ suboordinate to $%
\mathcal{G}$ such that $\{U_{1},V_{1}\}\subset \mathcal{G}_{1}$ and each
element of $Y$ is contained in at most two elements of the covering.

The local basis $\mathcal{S}_{a}$ shows $Z$ is not locally connected at $a$
(in fact $Z$ is not locally connected at $b$ either).

Let $A_{n}=\{a,b\}\cup (\{0\}\times \lbrack n,\infty ))$ and note $A_{n}$ is
compact and connected but $\cap _{n=1}^{\infty }A_{n}=\{a,b\}$.

\subsection{1-dimensional and generalized hereditarily unicoherent\label%
{1uni}}

Neither of the previous examples are generalized hereditarily unicoherent.
Theorem \ref{t2roid} ensures if $D$ is a 1-dimensional, non-Hausdorff,
generalized hereditarily unicoherent space, then $D$ is not locally
connected by continua and the example at hand is not locally connected.

Recall the previous example and the discussion of the subspace $D\subset Z$
such that $D=\{a,b\}\cup (\{0\}\times \lbrack 0,\infty )).$ Thus $D$ is a
connected, 1-dimensional, compact non-Hausdorff KC-space, $D$ is generalized
hereditarily unicoherent, but $D$ is not path connected or locally connected.

Lemma \ref{kcez} ensures $Z$ cannot be connected by continua (since
otherwise $Z$ would be $T_{2}$)

Let $A_{n}=\{a,b\}\cup (\{0\}\times \lbrack n,\infty ))$ and note $A_{n}$ is
compact and connected but $\cap _{n=1}^{\infty }A_{n}=\{a,b\}$.

\subsection{1-dimensional and locally contractible\label{1cont}}

To construct a 1-dimensional, non-Hausdorff $KC$ compactum $Y$ such that $Y$
is locally contractible, we glue together countably many closed rays at the
common minimal point, and then apply Alexandroff compactification to obtain $%
Y.$ Corollary \ref{ritecor} demands that $Y$ cannot be simply connected.

Let $\{e_{1},e_{2},..\}$ denote the standard unit vectors $%
e_{n}=(0,...0,1,0,0,...)$ in familiar $l_{2}$ Hilbert space (the space of
square summable sequences), and let $X$ denote the subspace of $l_{2}$
consisting of points of the form $\alpha e_{n}$ with $\alpha \in \lbrack
0,\infty ).$

Notice $X$ is a connected 1-dimensional metric space and $X$ fails to be
locally compact at precisely the point $(0,0,0,...).$ Thus by Lemma \ref%
{alex} if $Y=X\cup \{\infty \}$ denotes the Alexandroff compactification of $%
X$ then $Y$ is a non-Hausdorff, compact, connected KC-space.

Note $X$ is locally contractible. To check local contractibility at $\infty
, $ we manufacture a homotopy of $Y\backslash (0,0,0,..),$ shrinking basic
open neighborhoods $U$ of $\infty $ within themselves to $\infty ,$ as
follows.

Define $H:Y\backslash (0,0,0,...)\times \lbrack 0,\infty ]\rightarrow
Y\backslash (0,0,0,...)$ so that $H(\alpha e_{n},t)=(\alpha +t)e_{n}$ if $%
t<\infty $ and $H(x,t)=\infty $ otherwise.

To check that $H$ is continuous suppose $U$ is a subbasic open set in $%
Y\backslash (0,0,0,..).$

If $\infty \notin U$ and $U=(0e_{n},\alpha e_{n})$ then $H^{-1}(U)=(0e_{n},%
\alpha e_{n})$ which is open.

If $\infty \notin U$ and $U=(\alpha e_{n},\infty e_{n})$ then $%
H^{-1}(U)=(0e_{n},\infty e_{n})$ which is open.

If $\infty \in U$ we can assume $Y\backslash U$ is connected and $%
(0,0,0,...)\notin U$ and note $H^{-1}(U)=U$ which is open. Thus $H$ is
continuous.

To see informally why $Y$ is 1-dimensional, given an open covering $\mathcal{%
G}$ of $Y,$ observe there exist respective suboordinate basis elements $U$
and $V$ of the special points $\{(0,0,0,...),\infty \}$ such that $%
Y\backslash (U\cup V)$ is the countable union of disjoint closed line
segments $\beta _{1},\beta _{2},...$ and such that the respective endpoints
of $\partial \beta _{n}$ are respective limit points of $U$ and $V.$ All the
segments $\beta _{n}$ can be simultaneously lengthened slightly to pairwise
disjoint open arcs $\gamma _{1},\gamma _{2},..$ such that the ends of $%
\gamma _{n}$ are contained respectively in $U$ and $V.$ In particular
extending the standard construction that $\dim (\beta _{1}\cup \beta
_{2}..)=1$ relative to the covering $\mathcal{G}$ yields the desired
covering $\mathcal{G}_{1}.$

Let $J_{n}=[0e_{n},\infty e_{n})\cup \{\infty \}$ denote the loop containing 
$e_{n}$ with endpoints $(0,0,0,..)$ and $\infty $. Let $A_{N}=\cup
_{n=N}^{\infty }J_{n}$ and note $A_{N}$ is compact and connected but the two
point set $\cap _{N=1}^{\infty }A_{N}=\{(0,0,0,...),\infty \}$ is not
connected.

\subsection{1-dimensional, lcc, and simply connected\label{1lcc}}

In similar manner to the example from subsection \ref{1cont}, to construct a
1-dimensional, simply connected, non-Hausdorff $KC$ compactum $Y$ such that $%
Y$ is locally connected by continua we glue together countably many `long
lines' at the common minimal point, and then apply Alexandroff
compactification to obtain $Y.$

Let $L$ denote the (noncompact) long line (i.e. $L$ is obtained by attaching
open intervals between consecutive points of the minimal uncountable well
ordered set $S_{\Omega },$ to obtain a connected 1-dimensional nonseparable
space $L$ such that each point of $L$ has a neighborhood homeomorphic to $%
(0,1)$).

Let $z$ denote the minimal point of $L$ and let $X$ denote the quotient
space of $\{1,2,3,...\}\times L$ obtained by identifying $(m,z)$ and $(n,z).$
Thus we are gluing countably many copies of $L$ together at the minimal
point. Then $X$ is a 1st countable $T_{2}$ space which fails to be locally
compact. Thus if $Y$ is the Alexandroff compactification of $X$ then Lemma %
\ref{alex} ensures $Y$ is a connected, compact, non-Hausdorff KC-space.

To see why $\dim (Y)=1,$ we can apply essentially the same argument as in
subsection \ref{1cont}.

To see that $Y$ is simply connected note by construction $Y$ is 1
dimensional and contains no simple closed curves and hence by Lemma \ref{scc}
$Y$ is simply connected.

Let $A_{N}$ denote the union of the closed arcs $L_{N},L_{N+1},..$ and note $%
A_{N}$ is compact and connected but $\cap _{N=1}^{\infty }A_{N}$ is not
connected.

\subsection{2-dimensional, locally contractible, and n-connected for all n 
\label{2d}}

Theorem \ref{t2rite} forbids the existence of a 1-dimensional,
non-Hausdorff, locally contractible, $n$-connected, KC-space. Our
construction of such a space of dimension 2 is equivalent to taking a closed
disk, deleting a half open interval from the boundary, and then taking the
Alexandroff compactification of the remaining space. For convenient
coordinates, we begin with the closed first quadrant $X\subset R^{2}$ and
first attach a point $a$ to create a $T_{2}$ space which is not locally
compact at $a.$ We then take the Alexandroff compactification of $X\cup \{a\}
$ $\ $to obtain the desired space $Y=X\cup \{a\}\cup \{b\}.$The idea behind
the example is similar to the totally disconnected space constructed in
example 99 \cite{Steen}.

Intuition suggests $Y$ is \textbf{not} contractible but we do not settle
this question. However if $Y$ fails to be contractible, then $Y$ would serve
to highlight a potential difference between KC-spaces and the familiar
theory of absolute retracts, since every finite dimensional, $n$-connected,
compact, locally contractible metric space is necessarily contractible
(Theorem 4.2.33 \cite{vanmill}).

Let $X=[0,\infty )\times \lbrack 0,\infty )$ with the standard topology
(i.e. $X$ is the closed 1st quadrant in the Euclidean plane). We will attach
two points $a$ and $b$ to $X.$

Let $\mathcal{A}$ denote the collection of open sets of $X$ of the form $%
[0,\infty )\times (n,\infty )$ for $n\in \{1,2,3,...).$ Note $\mathcal{A}$
is countable, closed under finite intersections, and given $(x,y)\in X$
there exists an open set $U\subset X$ and $A\in \mathcal{A}$ such that $%
U\cap A=\emptyset .$ Let $a\notin X.$ If $\mathcal{S}_{a}$ denotes the
collection of sets of the form $\{a\}\cup A$ for $A\in \mathcal{S}_{a},$
Lemma \ref{w1} ensures that $X\cup \{a\}$ is a $T_{2}$ 1st countable space.
Observe $\{a\}\cup ([0,\infty )\times (n,\infty ))$ fails to have compact
closure in $X\cup \{a\}$ since the sequence $((0,n+1),(1,n+1),...)$ has no
subsequential limit in $X\cup \{a\}.$ Thus $X\cup \{a\}$ is not locally
compact at $a$. Let $b\notin X\cup \{a\}$ and let $Y=X\cup \{a\}\cup \{b\}$
denote the Alexandroff compactification of $X\cup \{a\}.$ Lemma \ref{alex}
ensures $Y$ is a compact, connected, non-Hausdorff KC-space.

To obtain a particular local basis $\mathcal{S}_{b}$ at $b,$ let $\mathcal{M}
$ denote the collection of nondecreasing maps $f:[0,\infty )\rightarrow
\lbrack 1,\infty ).$ For each $f\in \mathcal{M}$ let $U_{f}=\{(x,y)|f(y)<x%
\}. $ Note $U_{f}$ is open in $X.$ Let $\mathcal{B}$ denote the collection
of sets of the form $U_{f}$ and let $\mathcal{S}_{b}$ denote the collection
of sets of the form $\{b\}\cup U_{f}$ for $U_{f}\in \mathcal{B}$. To check
that $\mathcal{S}_{b}$ is a collection of open sets in $Y,$ let $\{b\}\cup
U_{f}\in \mathcal{S}_{b}$ and consider the complement $C=Y\backslash
(\{b\}\cup U_{f}).$ To check that $C$ is compact, given a covering $\mathcal{%
G}$ of $C$ by basic open sets in $Y$, obtain $U\in \mathcal{G}$ such that $%
a\in U.$ Notice $C\backslash U$ is compact in $X,$ since $C\backslash U$ is
a topological disk whose simple closed curve boundary is the concatenation
of 3 arcs: a line segment $\alpha \subset \partial U,$ an arc $\beta $
contained in $\partial U_{f},$ and a third arc $\gamma $ contained in the
union of the $x$ and $y$ axes in $X.$ Thus we can cover $C\backslash U$ by a
finite collection of the open sets in $\mathcal{G}\backslash \{U\}.$ Thus $%
\mathcal{S}_{b}$ is a collection of open sets in $Y$ each of which contains $%
b.$ To check that $\mathcal{S}_{b}$ is a local basis suppose we have a
compact set $C\subset Y\backslash \{b\}.$ Observe for each $n>0$ there
exists $m$ such that $([m,\infty )\times \lbrack 0,n])\cap C=\emptyset ,$ (
since otherwise there would exist a sequence $\{(x_{n},y_{n})\}\subset C$
such that $\{x_{n}\}$ is bounded and $y_{n}\rightarrow \infty $ and $%
\{(x_{n},y_{n})\}$ will be a noncompact closed subspace of the compact space 
$C).$ Hence we can manufacture $f\in \mathcal{M}$ and $U_{f}\in \mathcal{B}$
such that $C\subset Y\backslash U_{f}$ and, thus $\mathcal{S}_{b}$ is a
local basis at $b.$

To check that $\pi _{n}(Z)=0$ for all $n\geq 0$ it suffices to show if $Z$
is a compact $T_{2}$ space, then each map $f:Z\rightarrow Y$ is homotopic to
a constant, and the strategy is to show that both subspaces $X\cup \{a\}$
and $X\cup \{b\}$ are contractible and in particular $X\cup \{a\}$ admits
strong deformation retracts onto large compact spaces of $X\cup \{a\}.$
Normality of $Z$ combined with the standard pasting Lemma from general
topology will allow us to push $f$ into $X\cup \{b\},$ and then we will
homotop $f$ to the constant map $b,$ by contracting $X\cup \{b\}$ to the
point $b.$

To see that $X\cup \{a\}$ is homeomorphic to a closed topological disk with
a half open interval deleted from the boundary, let $h:[0,\infty
)\rightarrow \lbrack 0,1)$ be any homeomorphism and thus we can consider $X$
as the space $[0,1)\times \lbrack 0,1)$. Now consider the following
operations on the familiar closed unit square $[0,1]\times \lbrack 0,1].$
Take the quotient space by identifying $[0,1]\times \{1\}$ to a point and
note the quotient space $X_{1}$ is still a closed topological disk. Now
delete from $X_{1}$ the side $\{1\}\times \lbrack 0,1)$ to obtain the space $%
X_{2}.$ Then $X_{2}$ is homeomorphic to $X\cup \{a\}$ ( and $a$ corresponds
to the top side $[0,1]\times \{1\}$ of $X_{2}$).

Now suppose $Z$ is any compact $T_{2}$ space and $f:Z\rightarrow Y$ is any
map. Our first task is to homotop $f$ to a map $g$ such that $im(g)\subset
X\cup \{b\}.$

If $im(f)\subset X\cup \{b\}$ let $g=f.$ Otherwise let $A=f^{-1}(\{a\})$ and 
$B=f^{-1}(\{b\})$ and apply normality of $Z$ to obtain an open set $U\subset
Z$ such that $A\subset U$ and $\overline{U}\cap B=\emptyset .$ Let $\partial
U=\overline{U}\backslash U$ and note $f(\partial U)\subset X.$ Obtain a
compact topological disk $D\subset X$ such that $f(\partial U)\subset D$ and
obtain a strong deformation retract $H_{t}:X\cup a\rightarrow X\cup a$ onto $%
D.$ Let $J_{t}:X\rightarrow X\cup b$ denote the constant homotopy. Observe $(%
\overline{U}\times \lbrack 0,1])\cap ((Z\backslash U)\times \lbrack
0,1])\subset \partial U\times \lbrack 0,1].$ Now apply the pasting lemma 
\cite{Munkres} , gluing together the union of the restricted homotopies $%
H_{t}(f_{\overline{U}})\cup J_{t}(f_{Z\backslash U})$, and obtain a homotopy
of $f$ to a map $g=H_{1}(f_{\overline{U}})\cup f_{Z\backslash U}$ such that $%
g(Z)\subset X\cup \{b\}.$

Thus we have shown that $Y\backslash \{b\}$ is locally contractible, and
(for $n\geq 0$) any map $S^{n}:Z\rightarrow Y$ is homotopic in $Y$ to a map $%
g:S^{n}\rightarrow Y\backslash \{a\}.$ To complete the proof that $Y$ is $n$
connected and locally contractible it suffices to show $Y\backslash \{a\}$
is contractible and locally contractible. To accomplish this we will
manufacture a global contraction from $X\cup \{b\}$ to the $b$ whose
restrictions shrink basic open sets $\{b\}\cup U_{f}$ within themselves to $%
b.$

\bigskip Let $[0,\infty ]$ denote the one point compactification of $%
[0,\infty )$ and define a function $H:(X\cup \{b\})\times \lbrack 0,\infty
]\rightarrow X\cup \{b\}$ so that $H(x,y,t)=(x+t,y)$ if $(x,y)\in X$ and $%
t<\infty ,$ $H(b,t)=b$ for all $t\in \lbrack 0,\infty ],$ and $H(x,y,\infty
)=b.$ To check that $H$ is continuous it suffices to check that the preimage
under $H$ of subbasic sets is open\bigskip . Suppose $U\subset X\cup \{b\}$
is a subbasic open set. Let $J$ be open in $[0,\infty ).$ If $b\notin U$
note if $U=[0,x)\times J$ then $H^{-1}(U)=[0,x)\times J$ is open, and if $%
U=(x,\infty ]\times J$ then $H^{-1}(U)=[0,\infty )\times J$ is open. If $%
b\in U$ then let $U=\{b\}\cup B$ for $B\in \mathcal{B}$ and $B=U_{f}$ with $%
f\in \mathcal{M}$, and observe $H^{-1}(U)=U\times \lbrack 0,\infty ].$

To see informally why $\dim (Y)=2$ suppose $\mathcal{G}$ is an open covering
of $Y.$ Obtain subordinate basis elements $\{a\}\cup U$ and $\{b\}\cup U_{f}$
such that $U=[0,\infty )\times (n,\infty )$ and $U_{f}=\{(x,y)|f(y)<x\}$
with $f\in \mathcal{M}$. \ Let $D=Y\backslash (U\cup V).$ Notice $D$ is a
closed topological disk (as described in the earlier paragraph) and the
boundary point $z=(f(y),y)$ poses the `greatest risk' of belonging to too
many open sets in the cover $\mathcal{G}_{1}$ currently under construction,
since $(f(y),y)$ is a (unique) point of $D$ which is a limit point of both
sets $U$ and $V.$ So now cover $z$ by a tiny round open disk $V$ subordinate
to the original cover $\mathcal{G}$, and then proceed (as in a standard
proof that the topological disk $D\backslash U$ is 2 dimensional) to build
the desired cover $\mathcal{G}_{1}$ subordinate to $\mathcal{G}$.

Let $A_{N}$ denote the complement in $Y$ of $[0,n)\times \lbrack 0,n).$ Note 
$A_{N}$ is compact and connected but $\{a,b\}=\cap _{N=1}^{\infty }A_{N}.$

\subsection{2-dimensional, compact, non-Hausdorff, KC, contractible, and
locally contractible\label{2dnice}}

To create a 2-dimensional non-Hausdorff $KC$ compactum $Y$ such that $Y$ is
both contractible and locally contractible, apply directly Lemma \ref{ezc}
to the locally contractible example from subsection \ref{1cont}. \ Recall $%
J_{n}=[0e_{n},\infty e_{n})\cup \{\infty \}$ denotes the loop containing $%
e_{n}$ with endpoints $(0,0,0,..)$ and $\infty $. Let $A_{N}=(\cup
_{n=N}^{\infty }J_{n})\times \lbrack 0,1]$ and note $A_{N}$ is compact and
connected but the two point set $\cap _{N=1}^{\infty }A_{N}$ is not
connected.

\section{Supplemental Lemmas}

Justification of the theorems and examples throughout this paper relies on
various well known or elementary results concerning KC-spaces or basic
general topology. For convenience we include proofs.

The nested intersection of subcontinua indexed by an arbitrary ordered set $%
I $ is connected, and this yields a proof (in conjunction with Zorn's Lemma)
that any two points of a hereditarily unicoherent continuum are contained
within a canonical subcontinuum (Lemma \ref{uni}).

Familiar facts about locally path connected spaces (such spaces have open
components and such spaces are connected if they are path connected), have
counterparts with the notion of `path connected' replaced by `connected by a
continua', and this is the content of Lemmas \ref{q},\ref{lcc} and \ref{open}%
.

Lemmas \ref{w1} and \ref{alex} validate the basic properties of the main
examples of the paper at hand, and are essentially a special case of Theorem
5 \cite{Wil}.

Lemma \ref{ezc} yields a method to construct a contractible KC-space with
straightforward justification.

\begin{lemma}
\label{uni}Suppose the continuum $X$ is a hereditarily unicoherent continuum
and $\{a,b\}\subset X.$ Let $Y$ denote the intersection of all subcontinua
containing $\{a,b\}.$ Then $Y$ is connected.
\end{lemma}

\begin{proof}
To see that $Y$ is connected apply Zorn's Lemma as follows. Let $S$ denote
the collection of subcontinua of $X$ that contain $\{a,b\},$ partially
ordered by reverse inclusion such that $A\leq B$ if $B\subset A$. Since $Y$
is a $T_{2}$ space, the nested intersection of any linearly ordered
collection of subcontinua is connected, and thus every chain in $S$ has an
upper bound. By Zorn's Lemma let $M$ be a maximal element in $S.$ To see
that $M=Y$ note if $A\in S$ then $M\cap A\in S$ (since $X$ is \ hereditarily
unicoherent) and moreover $A\cap M\subset M.$ Thus $M=M\cap A$ since $M$ is
maximal. Thus $M\subset Y.$ Conversely, since $M\in S$ it follows that $%
Y\subset M.$
\end{proof}

\begin{lemma}
\label{q}Suppose $X$ and $Y$ are disjoint continua and $a\in X$ and $b\in Y$
and $Z$ is the quotient of $X\cup Y$ identifying $a$ and $b.$ Then $Z$ is a
continuum. Suppose $f:X\rightarrow W$ and $g:Y\rightarrow W$ are maps such
that $f(a)=g(b).$ Then $f\cup g:Z\rightarrow W$ is a continuous function.
\end{lemma}

\begin{proof}
Let $q:X\cup Y\rightarrow Z$ denote the quotient map. Since by definition $q$
is onto and continuous, $Z$ is compact and connected. To check $Z$ is $T_{2}$
suppose $\{c,d\}\subset Z.$ If $c\neq \{a,b\}$ and $d\neq \{a,b\}$ then
apply directly the $T_{2}$ property of $X\cup Y$ to obtain disjoint open
sets separating $c$ and $d.$

Suppose $x\in X\backslash \{a\}$. Apply the $T_{2}$ property of $X$ to
obtain open sets $U$ and $V_{a}$ in $X$ such that $U\cap V_{a}=\emptyset $
and $x\in U$ and $a\in V_{x}.$ Then $q^{-1}(U)$ and $q^{-1}(V_{a}\cup Y)$
separate $x$ and $\{a,b\}.$

By a symmetric argument if $y\in Y\backslash \{a,b\}$ then $y$ and $\{a,b\}$
can be separated in $Z.$

Notice $f\cup g$ is well defined. Recall $q$ is a quotient map and hence $%
f\cup g$ is continuous since if $z\in Z$ and $B=q^{-1}(z)$ then $(f\cup
g)_{|B}$ is constant.
\end{proof}

\begin{lemma}
\label{lcc}Suppose the nonempty connected space $X$ is locally connected by
continua. Then $X$ is connected by continua.
\end{lemma}

\begin{proof}
Fix $x\in X$ and let $A$ denote the set of $y$ in $X$ such that there exists
a continuum $Z_{y}$ and a map $f:Z_{y}\rightarrow X$ such that $%
\{x,y\}\subset im(f).$

To see that $A$ is open, given $y\in A$ obtain a continuum $Z_{y}$ and a map 
$f:Z_{y}\rightarrow X$ such that $\{x,y\}\subset im(f).$Apply the $lcc$
property of $X$ at $y$ and obtain an open set $U$ such $y\in U$ and $U$ is
connected by continua. Given $z$ in $U$ obtain a continuum $Z_{z}$ and a map 
$g:Z_{z}\rightarrow U$ such that $\{y,z\}\subset im(g).$ Obtain $a\in Z_{y}$
and $b\in Z_{z}$ such that $f(a)=y=g(b).$ Let $Z$ denote space obtained from 
$Z_{y}\cup Z_{z}$ by identifying $a$ and $b$. Apply Lemma \ref{q} to
conclude $Z$ is a continuum and the map $f\cup g:Z\rightarrow X$ satisfies $%
\{x,z\}\subset im(f\cup g).$ Thus $A$ is open.

To see that $A$ is closed suppose $z$ is a limit point of $A.$ Obtain an
open set $U_{z}$ such that $z\in U_{z}$ and $U_{z}$ is connected by
continua. Obtain $y\in A\cap U_{z}.$ Obtain disjoint continua $Z_{y}$ and $%
Z_{z}$ and maps $f:Z_{y}\rightarrow X$ and $g:Z_{z}\rightarrow X$ such that $%
\{x,y\}\subset im(f)$ and $\{y,z\}\subset im(g).$ Obtain $a\in Z_{y}$ and $%
b\in Z_{z}$ such that $f(a)=y=g(b).$ Let $Z$ denote space obtained from $%
Z_{y}\cup Z_{z}$ by identifying $a$ and $b$. Apply Lemma \ref{q} to conclude 
$Z$ is a continuum and the map $f\cup g:Z\rightarrow X$ satisfies $%
\{x,z\}\subset im(f\cup g).$ Thus $A$ is closed.

Since $X$ is connected and $A$ is both open and closed, $A=X.$ Thus $X$ is
connected by continua.
\end{proof}

\begin{lemma}
\label{open}Suppose the space $X$ is lcc and $Y$ is a component of $X.$ Then 
$Y$ is an open subspace of $X$ and $Y$ is connected by continua.
\end{lemma}

\begin{proof}
Suppose $y\in Y.$ Since $X$ is lcc obtain an open set $U$ such that $y\in U$
and such that $U$ is connected by continua. Then $U$ is connected (since $U$
is the union of images of continua, each of which contains the common point $%
y$) and hence $U\subset Y.$ Thus $Y$ is open and the open set $U$ also shows 
$Y$ is lcc. Hence by Lemma \ref{lcc} $Y$ is connected by continua.
\end{proof}

\begin{lemma}
\label{w1}Suppose $(X,\mathcal{T}_{X})$ is a connected $T_{2}$ space and $%
\mathcal{A}\subset \mathcal{T}_{X}$ such that$\mathcal{\ A\neq \emptyset }$
,and such that if $\{A_{1},A_{2}\}\subset \mathcal{A}$ then $A_{1}\cap
A_{2}\in \mathcal{A}$, and suppose $a\notin X$ and $Y=X\cup \{a\}.$ Suppose $%
\mathcal{S}_{a}$ denotes the collection of sets of the form $\{a\}\cup 
\mathcal{A}$ for $A\in \mathcal{A}$. Then $\mathcal{T}_{X}\cup \mathcal{S}%
_{a}$ is a basis for a topology $(Y,\mathcal{T}_{Y})$ on $Y$, $(Y,\mathcal{T}%
_{Y})$ is connected if $\emptyset \notin \mathcal{A}$, and $(Y,\mathcal{T}%
_{Y})$ is $T_{2}$ if for each $x\in X$ there exists $U\in \mathcal{T}_{X}$
and $A\in \mathcal{A}$ such that $U\cap A=\emptyset .$ If $X$ is first
countable and $\mathcal{A}$ is countable then $Y$ is first countable.
\end{lemma}

\begin{proof}
To check $\mathcal{T}_{X}\cup \mathcal{S}_{a}$ is a basis observe $\mathcal{T%
}_{X}\cup \mathcal{S}_{a}$ covers $Y$ since $\mathcal{A\neq \emptyset }$ and
hence $\mathcal{T}_{X}\cup \mathcal{S}_{a}$ is a subbasis. Thus sets of the
form $U_{1}\cap U_{2}...\cap U_{n}$ form a basis for a topology on $Y.$
Observe $\mathcal{T}_{X}\cup \mathcal{S}_{a}$ is closed under the operation
of finite intersections and thus $\mathcal{T}_{X}\cup \mathcal{S}_{a}$ is a
basis for $(Y,\mathcal{T}_{Y}).$ If $\emptyset \notin \mathcal{A}$ then $%
A=X\cap (\{a\}\cup A)\neq \emptyset $ for each basic open set $\{a\}\cup
A\in \mathcal{S}_{a}$ and thus $X$ is dense in $Y$ and hence $Y$ is
connected since $X$ is connected. Since $\mathcal{S}_{a}$ is a local basis
at $(Y,\mathcal{T}_{Y})$ is $T_{2}$ if for each $x\in X$ there exists $U\in 
\mathcal{T}_{X}$ and $A\in \mathcal{A}$ such that $U\cap A=\emptyset $ and
by also by definition $Y$ is first countable at $a$ if $\mathcal{S}_{a}$ is
countable.
\end{proof}

\begin{lemma}
\label{alex}Suppose $(X,\mathcal{T}_{X})$ is a 1st countable $T_{2}$ space
and $Y=X\cup \{b\}$ is the Alexandroff compactification of $X.$ Then $Y$ is
a compact KC-space. If $X$ is not locally compact then $Y$ is not Hausdorff.
If $X$ is connected and not compact, then $Y$ is connected.
\end{lemma}

\begin{proof}
Since $X$ is $T_{2}$ compact sets in $X$ are closed in $X$ and hence $X$ is
a $KC.$

To check that $Y$ is a KC-space suppose $C\subset Y$ and $C$ is not closed
in $Y$. If $b\in \overline{C}\backslash C$ then $C$ is not a compact
subspace of $X$ (since otherwise $Y\backslash C$ shows $b$ is not a limit
point of $C).$ Since open sets in $X$ are also open in $Y$, it follows that $%
C$ is not a compact subspace of $Y.$ If $x\in X$ and $x\in \overline{C}%
\backslash C,$ since $X$ is 1st countable, obtain a countable collection of
open sets $\{U_{n}\}\subset \mathcal{T}_{X}$ such that $...U_{3}\subset
U_{2}\subset U_{1}$ and if $V$ is an open in $X$ and $x\in V$ there exists $%
U_{N}$ such that $x\in U_{N}\subset V$ Since $x$ is a limit point of $C$
select $c_{n}\in C\cap U_{n}$ and note $c_{n}\rightarrow x$ and the set $%
A=\{x,c_{1},c_{2},...\}$ is compact in $X$. Thus $\{c_{1},c_{2},...\}$ is
not a closed subspace of $X$ and hence since $X$ is a KC-space $%
\{c_{1},c_{2},...\}$ is not compact in $X.$ Obtain an open cover $\mathcal{G}%
\subset \mathcal{T}_{X}$ of $\{c_{1},c_{2},...\}$ such that no finite
subcover of $\mathcal{G}$ covers $\{c_{1},c_{2},...\}.$ Observe $\mathcal{G}%
\cup \{Y\backslash A\}$ covers $C$ and has no finite subcover. Thus $C$ is
not compact in $Y$ and hence $Y$ is a KC-space.

If $X$ is not locally compact obtain $a\in X$ such that $X$ is not locally
compact at $a.$ To see that $Y$ is not Hausdorff suppose $a\in U$ and $b\in V
$ and each of $U$ and $V$ are open in $Y.$ To obtain a contradiction suppose 
$U\cap V=\emptyset .$ Then $b\notin U$ and hence $U\subset X.$ By definition
of $Y,$ $U$ is open in $X.$ Hence $\overline{U}\cap X$ is not compact since $%
X$ is not locally compact at $a.$ Thus, since $Y$ is a KC-space, $\overline{U%
}\cap X$ is not closed in $Y.$ Hence $b$ is a limit point of $\overline{U}%
\cap X$ and $b\in \overline{U}\backslash U.$ Since $b$ is a limit point of $U
$ we obtain the contradiction $V\cap U\neq \emptyset .$ This shows $X$ is
not $T_{2}.$

If $X$ is connected and not locally compact, then $X$ not compact, and thus $%
b$ is not an isolated point of $Y,$ and hence the connected subspace $X$ is
dense in $X\cup \{b\}$ and hence $Y$ is connected.
\end{proof}

\begin{lemma}
\label{ezc}If $X$ is metrizable and $Y=X\cup \{\infty \}$ denotes the
Alexandroff compactification of $X,$ then $Y\times \lbrack 0,1]$ is a
KC-space. If $Y\times \lbrack 0,1]$ is a KC-space and $Z$ denotes the
quotient of $Y$ such that $(x,1)\symbol{126}(y,1)$ then $Z$ is a
contractible KC-space and $Z$ is locally contractible at the point
determined by $Y\times \{1\}.$
\end{lemma}

\begin{proof}
To check that $Y\times \lbrack 0,1]$ is a KC-space suppose $C\subset Y\times
\lbrack 0,1]$ and suppose $C$ is not closed in $Y$ and suppose $z\in 
\overline{C}\backslash C.$ If $z\in X\times \lbrack 0,1]$ then, since $%
X\times \lbrack 0,1]$ is an open metrizable subspace of $Y\times \lbrack 0,1]
$, name a sequence $z_{n}\rightarrow z$ such that $z_{n}\in C.$ Note $%
\{z_{1},z_{2},..\}$ is a noncompact closed subspace of the space $C$ and
conclude $C$ cannot be compact. We have reduced to the case that $C\cap X$
is closed in $X$ and $z\in \{\infty \}\times \lbrack 0,1].$ If there exists $%
z_{n}\in (\{\infty \}\times \lbrack 0,1])\cap C$ such that $z_{n}\rightarrow
z$ then $C$ is not compact since $C\cap \{z_{1},z_{2},..\}$ is closed in $C$
and not compact. Thus to check the final case we assume $z$ is isolated in $%
C\cap \{\{\infty \}\times \lbrack 0,1]\}.$ Obtain a closed interval $[c,d]$
such that $\{z\}=(\{\infty \}\times \lbrack c,d])\cap C$ and $\{z\}\in
\{\infty \}\times (c,d).$ Let $D=C\cap (X\times \lbrack c,d])$ and observe $D
$ is a closed subspace of $C.$ Let $E$ denote the image of $D$ under the
natural projection $\Pi :D\rightarrow X.$ To see why $D$ cannot be compact,
note if $D$ were compact then $E$ is compact in $X$, and if $U=Y\backslash E$
then the open set $U\times (c,d)$ would show $z$ is not a limit point of $C,$
a contradiction. Thus since $D$ is a closed noncompact subspace of $C,$ we
conclude $C$ is not compact and hence $Y\times \lbrack 0,1]$ is a KC-space.

To check that $Z$ is a KC-space suppose $C\subset Z$ and suppose $C$ is not
closed. If $C\cap (Y\times \lbrack 0,1))$ is not closed in $Y\times \lbrack
0,1),$ then $C$ is not compact since $Y\times \lbrack 0,1)$ is a KC-space.
We have reduced to the case that $C$ is closed in $Y\times \lbrack 0,1).$ By
definition of the quotient topology, the natural quotient map $q:Y\times
\lbrack 0,1]\rightarrow Z$ shows $C$ is not closed in $Y\times \lbrack 0,1]$
and thus $C$ is not compact, and hence $q(C)=C$ is not compact in $Z.$ Thus $%
Z$ is a KC-space.

Ignoring specific properties of the spaces $Y$ and $Z,$ since $q$ is a
quotient map, the natural strong deformation retraction from $Y\times
\lbrack 0,1]$ onto $Y\times \{1\}$ induces a contraction of $Z$ and the
contraction shows that $Z$ is locally contractible at the special point.
\end{proof}

\section{Conclusions}

Theorems \ref{t2rite} and \ref{t2roid} show that every locally path
connected KC-space which contains no simple closed curve is $T_{2}$, and
every generalized hereditarily unicoherent KC-space which is locally
connected by continua is $T_{2}.$

It follows directly from Corollary \ref{ritecor} that every contractible,
locally contractible, 1-dimensional KC-space is $T_{2}.$ However subsection %
\ref{2dnice} shows there exists a 2-dimensional non-Hausdorff $KC$ compactum 
$Y$ such that $Y$ is both locally contractible and contractible.

The example $Y$ in subsection \ref{2d} is compact, 2-dimensional, locally
contractible, and $n$-connected for all $n$ but $Y$ has not been shown to be 
$Y$ contractible. If $Y$ fails to be contractible then $Y$ would amplify
another potential contrast between finite dimensional KC-spaces and finite
dimensional metric spaces (since every locally contractible $n-$connected
compact metric space is contractible).

For each non-Hausdorff KC-space $Y$ constructed in this paper there exists a
sequence of compact connected subspaces $...A_{3}\subset A_{2}\subset
A_{1}\subset Y$ such that $\cap _{n=1}^{\infty }A_{n}$ is not connected.

Examples are constructed in subsections \ref{1n}, \ref{1uni}, \ref{1cont},
and \ref{1lcc} of 1-dimensional non-Hausdorff KC-spaces in order to show
that slight weakening of the hypotheses of Theorems \ref{t2rite} and \ref%
{t2roid} can cause the $T_{2}$ conclusion to fail. The theorems and examples
lead naturally to the unsettled question of whether or not every
contractible 1-dimensional KC-space is Hausdorff.

\end{document}